\documentclass[12pt]{amsart}

\usepackage{amsmath,amssymb,amsthm}
\usepackage{hyperref}
\usepackage[utf8]{inputenc}
\usepackage{xcolor}
\usepackage{array}
\usepackage[margin=1in]{geometry}
\newtheorem{theorem}{Theorem}[section]
\newtheorem{lemma}[theorem]{Lemma}
\newtheorem{proposition}[theorem]{Proposition}
\newtheorem{corollary}[theorem]{Corollary}
\theoremstyle{definition}
\newtheorem{definition}[theorem]{Definition}
\newtheorem{example}[theorem]{Example}
\newtheorem{remark}[theorem]{Remark}

\newcommand{\C}{\mathbb{C}}
\newcommand{\R}{\mathbb{R}}
\newcommand{\Z}{\mathbb{Z}}

\newcommand{\HN}{H_N}

\title{Unitary Invariants of the Finite Heisenberg Group}
\author{Josh Katz}
\address{The MITRE Corporation}
\email{joshua.katz@mitre.org}
\date{\today}

\begin{document}
\maketitle
\begin{center}
{\small Approved for Public Release; Distribution Unlimited. Public Release Case Number 26-0143. Affiliation with the MITRE Corporation is for identification purposes only and is not intended to convey or imply MITRE’s concurrence with, or support for, the positions, opinions, or viewpoints expressed by the
author. \copyright 2026 The MITRE Corporation. ALL RIGHTS RESERVED.}
\end{center}

\begin{abstract}
Polynomial invariants of a group action often appear only in high degree, and in many representations the invariant ring imposes severe degree constraints before any nontrivial invariants can occur. In contrast, the larger class of unitary invariants---polynomials in both the variables and their conjugates---typically exhibits very different behavior, and their separating power is comparatively unexplored. 

We highlight this contrast in the setting of the finite Heisenberg group $\HN$. Although the polynomial invariant ring $\C[V]^{\HN}$ contains no nontrivial elements below degree $N$, we show that degree-six unitary invariants are already sufficient to separate generic $\HN$-orbits up to a global phase factor. These invariants arise from cubic equations involving the magnitudes of a vector and its discrete Fourier transform. A single polynomial invariant in degree $N$ then resolves the remaining global phase, yielding full generic orbit separation. Our proof utilizes fundamental results from phase retrieval \cite{beinert2018enforcing}. Along the way we will also explore the utility of unitary invariants in obtaining improved degree bounds for representations of cyclic groups.

This paper provides a concrete example in which the minimal separating degree for unitary invariants is dramatically lower than the minimal degree for polynomial invariants. This is a question which has come up in multiple works \cite{edidin2024orbit, bandeira2023estimation, edidin2025orbit, edidinreflection2026} but has thus far lacked conclusive results.
\end{abstract}

\section{Introduction}

Invariant theory traditionally focuses on the polynomial invariant ring $\C[V]^G$ associated to a linear action of a group $G$ on a complex vector space $V$. Classical questions ask for bounds on the number and degrees of generators for this ring. A more recent line of inquiry, initiated by Derksen and Kemper \cite{derksen2015computational}, concerns separating invariants: collections of polynomial invariants that distinguish all distinct $G$-orbits \cite{domokos2017degree, kohls2010degree}. Although these degree bounds are generally lower than the degree necessary to generate the invariant ring, they can still be quite large. Motivated by applied and computational questions, this has lead many authors to study the question of degree bounds necessary to separate generic orbits \cite{bandeira2023estimation, blum-smith2024degree, edidinreflection2026,edidin2025orbit}. Remarkably, in many cases of interest degree 3 invariants are sufficient to separate generic orbits.

When $V$ is a complex representation of a compact group, one may enlarge the algebra of invariants by allowing polynomials in both the variables and their conjugates. This yields the ring of \emph{unitary invariants}
\[
\C[x_1, \ldots, x_N, \bar{x}_1, \ldots, \bar{x}_N]^G,
\]
which is the ring of real algebraic functions on $V$ invariant under $G$. These invariants are ubiquitous in signal processing and harmonic analysis, yet their algebraic properties and separating power are far less understood than their polynomial counterparts. A central theme of this paper is that the minimal degrees needed to separate generic orbits using polynomial invariants and unitary invariants can differ dramatically.

\begin{remark}
Throughout this paper we use the term generic in the sense of real algebraic
geometry. In particular we say that a property holds for generic vectors in $\mathbb{C}^n$ if
there is an open set U whose complement is a real algebraic subset of dimension $m<n$ such
that the property holds on U. Such open sets are dense in the Euclidean topology and have full Lebesgue measure.
\end{remark}

\begin{definition}\label{def:gamma}
Let $\gamma(V, G)$ denote the smallest integer $d$ such that there exists a set of polynomial invariants of degree $\leq d$ that separates generic $G$-orbits. Likewise, let $\gamma^U(V, G)$ denote the smallest $d$ such that a set of unitary invariants of degree $\leq d$ separates generic orbits.
\end{definition}

The quantity $\gamma(V, G)$ is often extremely constrained by the representation theory of V (for abelian groups this will be dependent on the linear relations among the weights of the representation), while $\gamma^U(V, G)$ allows access to a larger class of invariants. In many cases, $\gamma^U(V, G)$ is much smaller than $\gamma(V, G)$.

\subsection{Motivating examples}

We begin with two classical examples in which polynomial invariants fail to separate orbits, yet low-degree unitary invariants succeed.

\begin{example}[Circle action on bandlimited functions]\label{ex:circle}
Let $G = S^1$ act diagonally on $V = L^2(S^1)_d=\{f\in L^2(S^1)|\hat{f}[j]=0, \forall |j|>d\}$. In the Fourier basis this action is represented by
\[
(\hat{f}[-d], \ldots, \hat{f}[d]) \mapsto (e^{-id\theta}\hat{f}[-d], \ldots, e^{id\theta}\hat{f}[d]).
\]
Since $\C[V]^{S^1} = \C[V]^{\C^*}$, any polynomial invariant can recover an orbit only up to $\C^*$; thus no finite-degree polynomial separating set exists. By contrast, the cubic unitary invariants
\[
\hat{f}[k]\hat{f}[s]\overline{\hat{f}[k+s]}, \quad -d \leq k, s \leq d,
\]
known as the bispectrum, determine the $S^1$-orbit of any $f$ with nonvanishing Fourier coefficients \cite{tukey1953spectral, smach2008generalized}. Hence $\gamma^U(V, S^1) = 3$ while $\gamma(V, S^1) = \infty$.
\end{example}

\begin{example}[Circle action with weights $1$ and $2$]\label{ex:weights}
Let $S^1$ act on $\C^2 = \C_1 \oplus \C_2$ via weights $1$ and $2$. There are no nonconstant polynomial invariants, yet the degree-three unitary invariants
\[
x_1\bar{x}_1, \quad x_2\bar{x}_2, \quad x_1^2\bar{x}_2
\]
separate all orbits with $x_1 x_2 \neq 0$. Again, $\gamma^U(V,S^1) = 3$ while $\gamma(V,S^1) = \infty$.
\end{example}

We give a general characterization of this behavior in the following proposition. 
\begin{proposition}\label{prop:finite-iff}
Let $G$ be a compact group with $V$ a complex representation. Then $\gamma(V, G)$ is finite if and only if $G$ is finite.
\end{proposition}

This result is standard in representation theory. A good refererence on complexifications of compact Lie groups can be found in the lecture notes \cite{bergeronreductivegroups}.
\begin{proof}
For compact $G$, we have $\C[V]^G = \C[V]^{G_\C}$ where $G_\C$ is the complexification. When $\dim G > 0$, the complexification $G_\C$ is a positive-dimensional algebraic group, and distinct $G$-orbits may lie in the same $G_\C$-orbit. In particular, the polynomial invariants cannot separate $G$-orbits when $G$ has positive dimension. Conversely, when $G$ is finite, $G_\C = G$ and classical results guarantee that polynomial invariants separate orbits.
\end{proof}

Remarkably, although $\gamma(V, G)$ will generally not be finite for compact groups with positive dimension, $\gamma^U(V, G)$ is always finite.

\begin{theorem}\label{thm:unitary-separate}
For any complex representation $V$ of a compact group $G$, the unitary invariants of $V$ separate $G$-orbits. In particular, $\gamma^U(V, G)$ is finite.
\end{theorem}

This is proved via the following lemma.

\begin{lemma}[\cite{bandeira2023estimation}]
Let G be a compact group and let $V$ be a finite dimensional representation over $\mathbb{R}$. The polynomial invariants separate G orbits of V.
\end{lemma}

\begin{proof}[Proof of Theorem \ref{thm:unitary-separate}]
The G-representation $V=\mathbb{C}^N$ can be viewed as a real representation of twice the dimension with coordinates $x_1,\overline{x_1}, x_2,\overline{x_2}, x_3, \overline{x_3},...,x_n,\overline{x_n}$ and hence the ring of unitary invariants can be viewed as the ring of polynomial invariants on $\mathbb{R}^{2n}$, $\mathbb{C}[x_1,...x_n,\overline{x_1},...,\overline{x_n}]^G=\mathbb{R}[\mathbb{R}^{2n}]^G\otimes_{\mathbb{R}} \mathbb{C}$. Hence we can view any vector as living in the $2n$-dimensional G-representation $v\in \mathbb{R}^{2n}$ and by the prior lemma, the orbit is determined by the polynomial invariants in $\mathbb{R}[\mathbb{R}^{2n}]^G$.
\end{proof}

\section{Cyclic Groups}

We now turn to finite groups, beginning with cyclic groups. Even in this simple setting, one can construct representations for which $\gamma^U(V,G) < \gamma(V,G)<\infty$.

\begin{example}\label{6}
Let $G = \Z_{6}$ with generator $\zeta_6$, a primitive 6th root of unity. Let G act on $V = \C_1 \oplus \C_2$ where the subscript denotes the weight of the character. i.e. $\zeta_6(v_1,v_2)=(\zeta_6v_1,\zeta_6^2v_2)$.

The only degree 3 invariant polynomial is $x_2^3$, so $\gamma(V, \Z_{6}) > 3$. However, the cubic unitary invariant $x_1^2\overline{x_2}$ allows us to solve for the orbit of a generic vector $(v_1,v_2)\in V$ using the following strategy:

Suppose we have two vectors $(v_1,v_2),(w_1,w_2)\in V$ with $$v_2^3=w_2^3, v_1^2\overline{v_2}=w_1^2\overline{w_2}.$$ Suppose further that $v_2\neq 0$.

This forces $w_2=\zeta_3v_2$ where $\zeta_3$ is an arbitrary cube root of unity. This in turn yields $w_1=\pm \zeta_3v_1=\zeta_6v_1$ and hence $(w_1,w_2)$ lies in the orbit of $(v_1,v_2)$.
\end{example}

We now extend this result to the general case of $\mathbb{Z}_{3n}$.
\begin{theorem}\label{thm:z3n}
Let $G = \Z_{3n}$ act on $V = \C_1 \oplus \C_2 \oplus \cdots \oplus \C_{n}$. Then $\gamma^U(V, G) = 3$ and $\gamma(V, G) > 3$.
\end{theorem}

\begin{proof}
The only polynomial invariant of degree $\leq 3$ is $x_{n}^3$, so $\gamma(V, G) > 3$. For the unitary bound, We use the following equations

$$S=\{x_1\overline{x_1}=r,x_1^2\overline{x_2}=a_1,x_1x_2\overline{x_3}=a_2, x_1x_3\overline{x_4}=a_3, ...,x_1x_{n-1}\overline{x_n}=a_{n-1},x_n^3=a_n\}.$$ Using the first equation $x_1\overline{x_1}=|x_1|=r$ we can assume without loss of generality that $x_1$ is fixed and has unit norm. 

Now note that if $x_1\neq 0$ then $x_2$ is uniquely determined by $x_1^2\overline{x_2}=a_1$. Going a step further, if $x_1,x_2\neq 0$ then $x_3$ is uniquely determined by $x_1x_2\overline{x_3}=a_2$. Continuing on in this way, fixing $x_1\in S^1$, the list of unitary invariant equations in S determine a vector $(v_1,...,v_n)\in V$ if $v_k\neq 0$ $\forall 1\leq k\leq n$. In other words, the invariants in S determine a generic $S^1$ orbit of V. 

More precisely, let $v=(v_1,...,v_n)\in V$ with $v_k\neq 0$ $\forall 1\leq k\leq n$ and suppose $w=(w_1,...,w_n)\in V$ satisfies $f(w)=f(v)$ $\forall f\in S$. We must have $w=(\theta v_1,\theta^2v_2,...,\theta^nv_n)$ with $|\theta|=1$. Now using the fact $w_n^3=v_n^3$ implies $(\theta^n v_n)^3=v_n$ which yields $\theta^{3n}=1$ and hence $w$ is in the $\mathbb{Z}_{3n}$ orbit of $v$.
\end{proof}

When the representation contains every character of $\mathbb{Z}_n$, the gap disappears.

\begin{theorem}[\cite{bandeira2023estimation}]\label{thm:regular}
Let $V = \C^n$ be the regular representation of $\Z_n$. Then
\[
\gamma(\C^n, \Z_n) = \gamma^U(\C^n, \Z_n) = 3.
\]
\end{theorem}

In the discrete Fourier basis, the degree-3 polynomial invariants are $$\{\hat{x}[i]\hat{x}[j]\hat{x}[N-i-j]\}_{0 \leq i,j \leq N-1}$$ and the unitary invariants include expressions like $$\{\hat{x}[i]\hat{x}[j]\overline{\hat{x}[i+j]}\}_{0 \leq i,j \leq N-1}.$$ The latter is the discrete analogue of the bispectrum.

This result has been extensively studied in the signal processing literature \cite{bendory2017bispectrum,chen2018spectral}.

The following generalization appears in \cite{edidin2025orbit}.

\begin{theorem}[\cite{edidin2025orbit}]\label{thm:contains-regular}
Let $V$ be a representation of a finite group over $\C$ which contains the regular representation. Then $\gamma(V, G) = \gamma^U(V, G) = 3$.
\end{theorem}

\section{The Heisenberg Group}

Let $\zeta = e^{2\pi i/N}$ and define the finite Heisenberg group
\[
\HN = \{(k, n, m) : k, n, m \in \Z_N\}
\]
with multiplication
\[
(k, n, m) \cdot (k', n', m') = (k + k', n + n', m + m' + kn').
\]
The group acts on $\C^N$ via:
\begin{itemize}
\item Time shifts: $T_k x_j = x_{j+k}$
\item Frequency modulations: $M_n x_j = \zeta^{nj} x_j$
\item Global phase: $Z_m x_j = \zeta^m x_j$
\end{itemize}
These operators satisfy the Heisenberg commutation relation $T_1 M_1 = \zeta M_1 T_1$.

This is the discrete analogue of the shift--modulation (Gabor) action underlying radar, sonar, and time--frequency analysis. The Heisenberg action arises naturally in joint time-delay and Doppler-shift estimation. For a good source on the utility of the finite Heisenberg group in some of these applications see \cite{howard2006finite, howard2006golay}.

\subsection{Absence of low-degree polynomial invariants}
Unlike for the standard representation of the cyclic group, the Heisenberg group has no low degree polynomial invariants.
\begin{proposition}
\label{thm:no-low-degree}
$\C[x_0, \ldots, x_{N-1}]^{\HN}$ contains no nonconstant invariants of degree $< N$.
\end{proposition}

\begin{proof}
The global phase operator $Z_m$ acts by $Z_m x_j = \zeta^m x_j$, multiplying each coordinate by $\zeta^m$. For a monomial $x_0^{a_0} \cdots x_{N-1}^{a_{N-1}}$ to be invariant under $Z_m$ for all $m \in \Z_N$, we need
\[
\zeta^{m(a_0 + \cdots + a_{N-1})} = 1 \quad \text{for all } m \in \Z_N.
\]
This requires $a_0 + \cdots + a_{N-1} \equiv 0 \pmod{N}$, so any invariant polynomial must have total degree divisible by $N$. In particular, there are no nonconstant invariants of degree $< N$.
\end{proof}

Moreover, the invariants that do appear in degree $\geq N$ are certainly not monomials and hence extracting a set of invariants which separate generic orbits is extremely challenging.

\section{The Heisenberg Bispectrum}

The key insight of this paper is that although the Heisenberg group does not admit low-degree polynomial invariants, its action on the \emph{magnitudes}
\[
y_j = |x_j|^2=x_j\overline{x_j}, \qquad z_k = |\hat{x}[k]|^2=\hat{x}[k]\overline{\hat{x}[k]}
\]
is much simpler: time shifts cyclically permute the $y_j$, and frequency modulations cyclically permute the $z_k$. Consequently, the bispectra of the real sequences $y,z\in \mathbb{R}^N$ are invariant under $\HN$.

Let $\hat{y}$ and $\hat{z}$ denote the discrete Fourier transforms of the real vectors $y$ and $z$.

\begin{definition}[Heisenberg bispectrum]\label{def:heisenberg-bispectrum}
For $x \in \C^N$, define the \emph{modulus bispectrum}
\[
B^M(x)(i,j) = \hat{y}[i]\hat{y}[j]\hat{y}[N - i - j]
\]
where $\hat{y}$ is the DFT of $y = (|x_0|^2, \ldots, |x_{N-1}|^2)\in \mathbb{R}^N$. Similarly, define the \emph{Fourier-modulus bispectrum}
\[
B^{FM}(x)(i,j) = \hat{z}[i]\hat{z}[j]\hat{z}[N - i - j]
\]
where $\hat{z}$ is the DFT of $z = (|\hat{x}[0]|^2, \ldots, |\hat{x}[N-1]|^2)\in \mathbb{R}^N$. The \emph{Heisenberg bispectrum} is the pair
\[
B^H(x) = (B^M(x), B^{FM}(x)).
\]
\end{definition}

\begin{remark}
The individual entries $B^M(x)(i,j)$ and $B^{FM}(x)(i,j)$ are unitary invariants of degree 3 in the real variables $y_0, \ldots, y_{N-1}$ and $z_0, \ldots, z_{N-1}$ respectively, and hence of degree 6 in the original variables $x_j, \bar{x}_j$.
\end{remark}

\begin{theorem}[Invariance]\label{thm:invariance}
For every $g \in \HN$ and $x \in \C^N$, we have $B^H(g \cdot x) = B^H(x)$.
\end{theorem}

\begin{proof}
We verify invariance under each generator of $\HN$:

\emph{Time shift $T_k$}: This cyclically permutes $(|x_0|, \ldots, |x_{N-1}|)$, so $y$ is permuted. The bispectrum of a real sequence is shift-invariant (since a cyclic shift multiplies Fourier coefficients by phases, which cancel in the bispectrum). The Fourier magnitudes $z$ are unchanged by $T_k$ since $|\widehat{T_k x}[j]| = |\hat{x}[j]|$.

\emph{Frequency modulation $M_n$}: This multiplies $x_j$ by $\zeta^{nj}$, leaving $|x_j|$ unchanged, so $y$ is unchanged. The Fourier transform satisfies $\widehat{M_n x}[k] = \hat{x}[k-n]$, so $z$ is cyclically permuted. Again, the bispectrum is shift-invariant.

\emph{Global phase $Z_m$}: This multiplies all coordinates by $\zeta^m$, leaving all magnitudes unchanged.
\end{proof}

\section{Separation of Heisenberg Orbits}

\subsection{Cyclic bispectrum theory}

We first recall the classical result for the cyclic group action.

\begin{theorem}[Cyclic bispectrum separation {\cite{bendory2017bispectrum, sadler1989shift}}]\label{thm:cyclic-bispectrum}
Let $y \in \R^N$ be a real signal with nonvanishing Fourier coefficients $\hat{y}[k] \neq 0$ for all $0 \leq k \leq N-1$. The bispectrum
\[
B(y)(i,j) = \hat{y}[i]\hat{y}[j]\hat{y}[N-i-j]
\]
determines $y$ up to a cyclic shift.
\end{theorem}

\subsection{Recovery of magnitude data}

\begin{lemma}\label{lem:generic-nonvanishing}
For generic $x \in \C^N$, both $y = (|x_0|^2, \ldots, |x_{N-1}|^2)$ and $z = (|\hat{x}[0]|^2, \ldots, |\hat{x}[N-1]|^2)$ have all Fourier coefficients nonvanishing.
\end{lemma}
Note that throughout this work, ``generic'' means outside a fixed proper real algebraic subset of $\C^N$.
\begin{proof}
The condition $\hat{y}[k] = 0$ defines a real algebraic hypersurface in the space of signals. Similarly for $\hat{z}[k] = 0$.
\end{proof}

\begin{theorem}[Separation up to shift]\label{thm:separation-shift}
For generic $x \in \C^N$, the Heisenberg bispectrum $B^H(x)$ determines:
\begin{enumerate}
\item the magnitude vector $y_j = |x_j|$ up to a cyclic shift, and
\item the Fourier-magnitude vector $z_k = |\hat{x}[k]|$ up to a cyclic shift.
\end{enumerate}
\end{theorem}

\begin{proof}
By Lemma~\ref{lem:generic-nonvanishing}, generic signals have nonvanishing Fourier coefficients for both $y$ and $z$. By Theorem~\ref{thm:cyclic-bispectrum}, the bispectra $B^M$ and $B^{FM}$ then determine $y$ and $z$- the squares of the moduli- each up to a cyclic shift. This in turn determines the moduli directly, $|x[j]|, |\hat{x}[k]|$ up to cyclic shift.
\end{proof}

\subsection{Reduction to classical phase retrieval}

The following classical result connects magnitude measurements to signal recovery.

\begin{theorem}[1-D phase retrieval {\cite{beinert2018enforcing}}]\label{thm:phase-retrieval}
Let $x \in \C^N$ be a generic vector. If another vector $x'$ satisfies
\[
|x_j| = |x'_j| \quad \text{and} \quad |\hat{x}[k]| = |\hat{x}'[k]| \quad \text{for all } j, k,
\]
then $x' = \lambda x$ for some $\lambda \in S^1$.
\end{theorem}
For the exact genericity condition see \cite{beinert2018enforcing}.

\begin{corollary}[Separation up to global phase]\label{cor:separation-phase}
For generic $x \in \C^N$,
\[
B^H(x) = B^H(x') \implies x' = \lambda \cdot (k_0, n_0, 0) \cdot x \quad \text{for some } \lambda \in S^1, \, k_0, n_0 \in \Z_N.
\]
That is, the Heisenberg bispectrum determines the $\HN$-orbit up to a global phase.
\end{corollary}

\begin{proof}
By Theorem~\ref{thm:separation-shift}, the bispectrum determines $(|x_j|)$ and $(|\hat{x}[k]|)$ up to cyclic shifts. That is, there exist $k_0, n_0 \in \Z_N$ such that
\[
|x'_j| = |x_{j + k_0}|, \qquad |\hat{x}'[k]| = |\hat{x}[k + n_0]|
\]
for all $j, k$. Let $x'' = (k_0, n_0, 0) \cdot x$. Then $|x''_j| = |x'_j|$ and $|\hat{x}''[k]| = |\hat{x}'[k]|$ for all $j, k$. By Theorem~\ref{thm:phase-retrieval}, $x' = \lambda x''$ for some $\lambda \in S^1$.
\end{proof}
In order to fully recover the orbit we must force $\lambda\in S^1$ to be an N-th root of unity.
\subsection{Eliminating the global phase}

The magnitude data is invariant under multiplication by a global phase $e^{i\theta}$, so the bispectrum-type invariants cannot distinguish between $x$ and $e^{i\theta}x$. We now show that a single degree-$N$ polynomial invariant resolves this ambiguity.

\begin{definition}
Define the degree-$N$ polynomial invariant
\[
I_N(x) = x_0^N + x_1^N + \cdots + x_{N-1}^N.
\]
\end{definition}

\begin{lemma}\label{lem:IN-invariant}
$I_N$ is invariant under $\HN$.
\end{lemma}

\begin{proof}
For time shifts: $I_N(T_k x) = \sum_j x_{j+k}^N = \sum_j x_j^N = I_N(x)$ by reindexing.

For frequency modulations: $I_N(M_n x) = \sum_j (\zeta^{nj} x_j)^N = \sum_j \zeta^{Nnj} x_j^N = \sum_j x_j^N = I_N(x)$ since $\zeta^N = 1$.

For global phase: $I_N(Z_m x) = \sum_j (\zeta^m x_j)^N = \zeta^{Nm} \sum_j x_j^N = I_N(x)$ since $\zeta^{Nm} = 1$.
\end{proof}

\begin{theorem}[Full orbit separation]\label{thm:full-separation}
Let $x \in \C^N$ be generic. If $x'$ satisfies
\[
B^H(x) = B^H(x') \quad \text{and} \quad I_N(x) = I_N(x'),
\]
then $x'$ lies in the same $\HN$-orbit as $x$.
\end{theorem}

\begin{proof}
By Corollary~\ref{cor:separation-phase}, $x' = \lambda \cdot g \cdot x$ for some $\lambda \in S^1$ and $g \in \HN$. Since $I_N$ is $\HN$-invariant, we have
\[
I_N(x') = I_N(\lambda \cdot g \cdot x) = \lambda^N I_N(g \cdot x) = \lambda^N I_N(x).
\]
The condition $I_N(x') = I_N(x)$ thus requires $\lambda^N = 1$, so $\lambda = \zeta^m$ for some $m \in \Z_N$. But then $x' = \zeta^m \cdot g \cdot x = Z_m \cdot g \cdot x$, which lies in the $\HN$-orbit of $x$.
\end{proof}

\subsection{Small example with $N=3$}

Let $x=(x_0,x_1,x_2)\in \mathbb{C}^3$ and let $\hat{y}[k]=|x_0|^2+\zeta_3^k|x_1|^2+\zeta_3^{2k}|x_2|^2$ for $k=0,1,2$. We have the following modulus unitary invariants of the form
\begin{align*}
x_0\overline{x_0}+x_1\overline{x_1}+x_2\overline{x_2}=\hat{y}[0],\\
(x_0\overline{x_0}+\zeta_3x_1\overline{x_1}+\zeta_3^2x_2\overline{x_2})*(x_0\overline{x_0}+\zeta_3^2x_1\overline{x_1}+\zeta_3x_2\overline{x_2})=\hat{y}[1]\hat{y}[2]\\
B^M(x)(1,1)=(x_0\overline{x_0}+\zeta_3x_1\overline{x_1}+\zeta_3^2x_2\overline{x_2})^3=\hat{y}[1]^3
\end{align*}
These invariants determine $(|x_0|,|x_1|,|x_2|)\in \mathbb{R}^3$ up to cyclic shift.

Similarly, we let $\hat{z}[k]=|\hat{x}[0]|^2+\zeta_3^k|\hat{x}[1]|^2+\zeta_3^{2k}|\hat{x}[2]|^2$ to obtain the Fourier modulus invariants
\begin{align*}
\hat{x}[0]\overline{\hat{x}[0]}+\hat{x}[1]\overline{\hat{x}[1]}+\hat{x}[2]\overline{\hat{x}[2]}=\hat{z}[0],\\
(\hat{x}[0]\overline{\hat{x}[0]}+\zeta_3\hat{x}[1]\overline{\hat{x}[1]}+\zeta_3^2\hat{x}[2]\overline{\hat{x}[2]})*(\hat{x}[0]\overline{\hat{x}[0]}+\zeta_3^2\hat{x}[1]\overline{\hat{x}[1]}+\zeta_3\hat{x}[2]\overline{\hat{x}[2]})=\hat{z}[1]\hat{z}[2]\\
B^{FM}(x)(1,1)=(\hat{x}[0]\overline{\hat{x}[0]}+\zeta_3\hat{x}[1]\overline{\hat{x}[1]}+\zeta_3^2\hat{x}[2]\overline{\hat{x}[2]})^3=\hat{z}[1]^3
\end{align*}

Combining these with the polynomial invariant $$x_0^3+x_1^3+x_2^3$$ completes the generic separating set.

\section{Summary and Discussion}

\begin{theorem}[Main result]\label{thm:main}
For the Heisenberg action on $\C^N$ with $N > 3$:
\begin{enumerate}
\item There are no nonconstant polynomial invariants of degree $< N$.
\item Degree-6 unitary invariants (the Heisenberg bispectrum) separate generic orbits up to a global phase.
\item Adding the single degree-$N$ invariant $I_N$ achieves full separation.
\end{enumerate}
Thus $\gamma^U(\C^N, \HN)\leq N$ while $\gamma(\C^N, \HN) \geq N$.
\end{theorem}
\begin{remark}
It is unknown whether $\gamma(\C^N, \HN)= N$, though one should keep in mind that the polynomial invariants are much more limited and hence a constructive result resembling the unitary case is unlikely. 
\end{remark}

The Heisenberg bispectrum exhibits a new phenomenon: \emph{separating invariants may exist in low degree even when the invariant ring has no low-degree polynomial invariants}. The construction relies fundamentally on results from phase retrieval, which studies when signals can be recovered from magnitude measurements \cite{beinert2018enforcing}. The classical result that magnitude and Fourier magnitudes generically determine a signal up to a global phase forms the bridge between our bispectral invariants and orbit recovery.

\subsection{Open problems}

\begin{enumerate}
\item We prove $\gamma^U(\C^N, \HN) \leq N$ but can we reduce this degree bound even more? Most of the work is being done by invariants of degree 6.

\item Can the method of studying unitary invariants give improved degree bounds for orbit recovery for other groups?

\end{enumerate}


\bibliographystyle{plain}
\bibliography{ref}
\end{document}